\documentclass{amsart}

\newtheorem{theorem}{Theorem}[section]
\newtheorem{lemma}[theorem]{Lemma}
\newtheorem{cor}[theorem]{Corollary}
\newtheorem{prop}[theorem]{Proposition}
\usepackage{enumerate,amssymb}
\theoremstyle{definition}

\theoremstyle{remark}

\numberwithin{equation}{section}

\def\a{\alpha}
\def\b{\beta}
\def\g{\gamma}

\def\|{\parallel}
\def\<{\langle}
\def\>{\rangle}
\def\R{\mathbb{R}}
\def\s{\sharp}
\newcommand{\B}{\mathbb B}
\newcommand{\A}{\mathcal A}
\newcommand{\M}{\mathcal M}
\newcommand{\Mn}{\mathbb M_n}
\newcommand{\Md}{\mathbb M_d}
\renewcommand{\H}{\mathcal H}
\renewcommand{\le}{\leqslant}
\renewcommand{\leq}{\leqslant}
\renewcommand{\ge}{\geqslant}
\renewcommand{\geq}{\geqslant}
\newcommand{\C}{\mathbb C}

\begin{document}

\title{ An asymmetric Kadison's inequality  }

\author{Jean-Christophe Bourin}
\address{Laboratoire de math\'ematiques,
 Universit\'e de Franche-Comt\'e, 16 route de Gray,
 25030 Besan\c{c}on, Fran\-ce}
\email{jcbourin@univ-fcomte.fr}


\author{\'Eric Ricard}
\address{Laboratoire de math\'ematiques,
 Universit\'e de Franche-Comt\'e, 16 route de Gray,
 25030 Besan\c{c}on, Fran\-ce}
\email{eric.ricard@univ-fcomte.fr}


\date{ February 28, 2010.}

\thanks{Research of both authors was supported by ANR 06-BLAN-0015}

\begin{abstract} Some  inequalities for  positive
linear maps on matrix algebras are given, especially asymmetric
extensions of  Kadison's inequality and  several operator versions
of Chebyshev's inequality. We also discuss well-known results around
the matrix geometric mean and connect it with complex interpolation.
\end{abstract}

\maketitle

\section*{ Introduction }

\medskip

This note lies in the scope of matricial inequalities. The main
motivation of this theory is to extend some classical inequalities
for reals to self-adjoint matrices. Of course, the non-commutativity
of $\Mn$ (the space of $n\times n$ complex matrices) enters into the
game, making things much more complicated. The book \cite{Bha} is a
very good introduction to this subject. Many techniques have been
developed, such as the theory of operator monotone/convex functions
and their links with completely positive maps.  Nevertheless  the
proofs very often rely on quite clever but simple arguments. As an
illustration of the available tools, a very classical result is
Kadison's inequality \cite{K} saying that if $\Phi:\A\to \M$ is a
unital positive (linear) map between $C^*$-algebras, then for a
self-adjoint element $A$ in $\A$,
$$\Phi(A)^2\leq \Phi(A^2).$$ Taking $\Phi: \Mn\oplus \Mn\to \Mn$,
 $\Phi(A,B)=(A+B)/2$, this reflects the operator convexity of $t\mapsto
 t^2$. One can think of it as a kind of Jensen's or Cauchy-Schwarz's
 inequality. The main
 motivation of this paper is to try to get comparison relations
 between the images of the powers of $A$. At a first glance, one does
 not expect to have many positive results beyond operator
 convexity. But surprisingly, we notice here that
 $|\Phi(A^p)\Phi(A^q)|\leq \Phi(A^{p+q})$ provided that $0\leq p\leq
 q$ and $A\geq 0$. This and some variations are our concern of the first
 section.

The second section deals more generally with monotone pairs,
 in place of pairs $(A^p, A^q)$. These are  pairs
 $(A,B)$ of positive operators, characterized by joint relations
 $A=f(C)$ and $B=g(C)$ for some $C\ge 0$ in $\Mn$ and two
  non-decreasing, non-negative functions $f(t)$ and $g(t)$ on
 $[0,\infty)$. Comparing $\Phi(A)\Phi(B)$ with $\Phi(AB)$ is
 non-commutative versions of the classical
 Chebyshev's inequality,
$$\Big(\frac 1 n\sum_{i=1}^n a_i\Big)\cdot\Big( \frac 1 n\sum_{i=1}^n b_i\Big)\leq\frac 1 n \sum_{i=1}^n a_ib_i
$$ for non-negative increasing sequences $\{a_i\}$ and $\{b_i\}$ (here
$\Phi$ is just a state).

In the last part, we point out the links between complex
interpolation and power means. Along with some very classical
approach and an idea of \cite{FK}, this  is used to furnish a simple
proof of Furuta's inequality, which is the main tool in section 1.

We assume that the reader is familiar with basic notions in operator
and matricial inequalities theories. When possible, we state the
results in their general context, that is, for von Neumann or
$C^*$-algebras. But matrix inequalities for positive linear maps are
essentially finite dimensional results, especially when it comes to
unitary congruences. So, the reader may like to think of the
algebras as $\Mn$.

\section{Kadison's asymmetric type inequalities}

In this section, we deal with   positive linear maps $\Phi :
{\mathcal{A}}\to{\mathcal{M}}$ between two unital $C^*$-algebra $\A$
and   $\M$ with units denoted by $I$. In fact, we may assume that
${\mathcal{A}}$ is the unital $C^*$-algebra  generated by a single
positive operator $A$; hence, by a classical dilation theorem of
Naimark (see \cite{Paulsen}, Theorem 3.10), our maps $\Phi$ will be
automatically  completely positive. We will also always assume that
these maps are unital, $\Phi(I)=I$, or more generally
sub-unital, $\Phi(I)\le I$.

Kadison's inequality is one
 of the most basic and useful results for such sub-unital  maps; it states
that for any $A\in \A^{sa}$ (the self-adjoint elements in $\A$),
$$\Phi(A)^2 \le \Phi(A^2).$$
More generally, if  $f$ is operator convex on an interval containing
$0$ and $f(0)\le0$, then one has
$$f(\Phi(A))\leq \Phi(f(A))$$
for all   $A\in \A^{sa}$ with spectrum in the domain of $f$. If we
drop the condition that $0$ is in the domain of $f$ and $f(0)\le0$,
  this Jensen's inequality remains true for unital maps.
 When  $\Phi$ is the compression map to a subspace,   it
 is then a basic characterization of operator convexity due to Davis
\cite{Davis}. The general case was noted in an influential paper of
Choi \cite{C}; nowadays everything is very clear using Stinespring's
theorem (see \cite{Paulsen}, Theorem 4.1) for completely positive maps.

 First examples of operator convex/concave functions
on $\R^+$ are given by powers, we refer to the corresponding Jensen
inequalities as Choi's inequality; for $A\in \A^+$ (the positive
cone of $\A$),
$$\Phi(A^p)\le\Phi(A)^p, \quad 0\le p\le 1,
$$
and
$$
\Phi(A)^p\le\Phi(A^p), \quad 1\le p\le 2.$$ In the spirit of
operator convexity, one can naturally think of looking for more
comparison relations between powers of $A$.

We start with an asymmetric extension of Kadison's inequality :

\medskip
 \begin{theorem}\label{kad}
   Let $A\in \A^+$ and\, $0\le p\le q$. Then,
$$
|\Phi(A^p)\Phi(A^q)| \le \Phi(A^{p+q}).
$$
\end{theorem}

\medskip
\begin{proof} We will derive this result from Furuta's inequality that we
recall as follows : {\it Let $X\ge Y\ge 0$ in some $\B(\H)$,
 let $\a\ge 1$ and $\b\ge 0$. Then, for
$\g\ge (\a+2\b)/(1+2\b)$,
$$
X^{(\a+2\b)/\g} \ge (X^\b Y^\a X^\b)^{1/\g},
$$
with equality if and only if $X=Y$.}

\noindent We will discuss about it in Section 3.
 Now, set
$$
X=\Phi(A^q)^{\frac{p}{q}}, \qquad Y=\Phi(A^p).
$$
By Choi's inequality, $X\ge Y$. Then we apply Furuta's inequality to
$X$ and $Y$ with
$$
\a=2,\quad \b=\frac{q}{p},\quad \g=2.
$$
(Note that
$$ \g =2\ge \frac{2+2(q/p)}{1+2(q/p)}=\frac{\a
+2\b}{1+2\b}
$$
so that assumptions of Furuta's inequality are satisfied.) Thus we
obtain
$$
\{\Phi(A^q)^{\frac{p}{q}}\}^{\frac{2+2(q/p)}{2}} \ge
 \left(
\{\Phi(A^q)^{\frac{p}{q}}\}^{\frac{q}{p}}\{\Phi(A^p)\}^2\{\Phi(A^q)^{\frac{p}{q}}\}^{\frac{q}{p}}
\right)^{1/2},
$$
equivalently
\begin{equation}
\Phi(A^q)^{1+p/q}\ge |\Phi(A^p)\Phi(A^q)|.
\end{equation}
Since $1\le 1+p/q\le 2$, using once again Choi's inequality for
operator convex functions,
\begin{equation}
\Phi(A^{p+q})\ge\Phi(A^q)^{1+p/q}.
\end{equation}
Combining (1.1) and (1.2) completes  the proof.
\end{proof}

\medskip

\noindent {\bf Remark}. Actually, we have shown the stronger
inequality (1.1) that can be restated as follows : For $0\le\a\le
1$, $\Phi(A)^{1+\a}\ge |\Phi(A^\a)\Phi(A)|$.

\medskip
\noindent {\bf Remark.} For $0<p\le q$, the equality case in Theorem
\ref{kad} entails the equality case in Choi's inequality, so that
$\Phi(A^t)=\Phi(A)^t$ for all $t>0$, in other words $A$ is in the
multiplicative domain of $\Phi$.

 \medskip

\begin{cor}\label{kadinv} Assume that moreover $\M$ is a von Neumann algebra,
then for $A\in \A^+$ and  $p,\,q\geq 0$, there is a partial isometry
$V\in \M$  such  that
$$
|\Phi(A^p)\Phi(A^q)| \le V\Phi(A^{p+q})V^*.
$$
If $\M$ is finite, then $V$ can be chosen to be unitary.
\end{cor}

This follows from Theorem \ref{kad} and the polar decomposition.
Indeed, for any $Z\in \M$, there is a partial isometry $V$ so that
$Z=V|Z|$ and moreover $|Z^*|=V|Z|V^*$ and $|Z|=V^*|Z^*|V$. If $\M$
is finite, then $V$ can also be chosen unitary. But in general, one
can not assume $V$ to be unitary.

\medskip
\renewcommand{\epsilon}{\varepsilon}
\noindent {\bf Remark.}  Fix $0< q<p$. Let
$$A=\left[
\begin{array}{ccc}
1+\epsilon &\epsilon& \epsilon \\ \epsilon &2+\epsilon& \epsilon \\
\epsilon &\epsilon& 3+\epsilon
\end{array}\right], \qquad \qquad
B=\left[ \begin{array}{ccc} 1 &1 & \frac 1 2\\  1 & 1 &\frac 1 2
\\ \frac 1 2 & \frac 1 2 & 1 \end{array}\right]$$
and let $\Phi$ be the Schur product with $B$. Then for $\epsilon$
small enough, it follows from tedious computations on derivatives,
that we can not get rid of $V$ in Corollary \ref{kadinv} like in
Theorem \ref{kad}.

\medskip

From now on, we come back to the setting of matrix inequalities where
$\M=\Mn$ for some positive integer $n$.

The next two results are variations of Corollary \ref{kadinv}. We
rely on an easy consequence of the min-max principle; if $A\ge B\ge
0$ in $\Mn$ and $f(t)$ is non-decreasing on $[0,\infty[$, then
$f(A)\ge Vf(B)V^*$ for some unitary $V\in \Mn$.

\begin{prop} Let $A\in \A^+$ and $p,q,r\ge 0$ such
that $\min\{p,r\}\le q/2$ and $\max\{p,r\} \le q$. Then, for some
unitary $V\in \Mn$,
$$
|\Phi(A^p)\Phi(A^q)\Phi(A^r)| \le V\Phi(A^{p+q+r})V^*.
$$
\end{prop}

\begin{proof} We may assume $q=1$ and $r\le
1/2$. We then have
$$
\Phi(A^r)^2\le \Phi(A^{2r}) \le \Phi(A)^{2r}
$$
so that $\Phi(A^r) =\Phi(A)^rK$ for a contraction $K$. Hence, for
some unitary $U$,
\begin{equation*}
|\Phi(A^p)\Phi(A)\Phi(A^r)| \le U|\Phi(A^p)\Phi(A)^{1+r}|U^*
\end{equation*}
so
\begin{equation}
|\Phi(A^p)\Phi(A)\Phi(A^r)|\le
U\left(\{\Phi(A)\}^{1+r}\{\Phi(A^p)\}^2\{\Phi(A)\}^{1+r}\right)^{\frac{1}{2}}U^*.
\end{equation}
Now, observe that a byproduct of Furuta's inequality is:

{\it If $X\ge Y\ge 0$ and, $\a,\,\b\ge 0$, then for some unitary
$W$,
$$
X^{\a+2\b}\ge W(X^{\b}Y^{\a}X^{\b})W^*.
$$
} Applying this inequality to $X=\Phi(A)^p$ and $Y=\Phi(A^p)$ with
$\a=2, \b=(1+r)/p$ and combining with (1.3) yields
\begin{equation*}
|\Phi(A^p)\Phi(A)\Phi(A^r)|\le V_0(\Phi(A)^{1+p+r})V_0^*
\end{equation*}
for some unitary $V_0$. Since, by a byproduct of Choi's inequality, we also have some unitary $V_1$ such that
\begin{equation*}
\Phi(A)^{1+p+r}\le V_1\Phi(A^{1+p+r})V_1^*,
\end{equation*}
we get the conclusion.
\end{proof}

 At the cost of one more unitary congruence, assumptions of
Proposition 1.3 can be relaxed. We will use an inequality of Bhatia
and Kittaneh (see \cite{Bha} for an elementary proof) : For all $A,\,B$
in some finite von Neumann algebra $\M$, there is some unitary $U\in
\M$ such that
$$
|AB^*|\le U\frac{|A|^2+|B|^2}{2}U^*.
$$

\begin{prop} Let $A\ge 0$ in $\A$ and let $
p,q,r\ge 0$ with $q\ge p,r$. Then, for some unitaries $U,\,V$ in
$\Mn$,
$$
|\Phi(A^p)\Phi(A^q)\Phi(A^r)| \le
\frac{U\Phi(A^{p+q+r})U^*+V\Phi(A^{p+q+r})V^*}{2}.
$$
\end{prop}

\begin{proof} We may assume $q=1$. Let
$\a\in[0,1]$ and note that by Bhatia-Kittaneh's inequality,
\begin{align}
|\Phi(A^p)\Phi(A)\Phi(A^r)| &= |\Phi(A^p)\Phi(A)^{\a}\cdot\Phi(A)^{1-\a}\Phi(A^r)| \notag \\
&\le
W\frac{|\Phi(A^p)\Phi(A)^{\a}|^2+|\Phi(A^r)\Phi(A)^{1-\a}|^2}{2}W^*
\end{align}
for some unitary $W$.  Then set $\a=(r-p+1)/2$ (hence $0\leq\a\leq
 1$). We may estimate each summand in (1.4) via Furuta's inequality,
 since $\Phi(A^p)\le \Phi(A)^p$ and $\Phi(A^r)\le \Phi(A)^r$. For the
 first summand, there are some unitaries $W_0$ and $W_1$ such that
 \begin{align}
 |\Phi(A^p)\Phi(A)^{\a}|^2 &=
 \{\Phi(A)^p\}^{\frac{r-p+1}{2p}}\{\Phi(A^p)\}^2
 \{\Phi(A)^p\}^{\frac{r-p+1}{2p}}\notag \\
 &\le W_0\Phi(A)^{1+p+r}W^*_0 \notag \\
  &\le W_1\Phi(A^{1+p+r})W_1^*
  \end{align}
  where the last step follows from Choi's inequality. We also have a unitary
  $W_2$ such that
  \begin{equation}
  |\Phi(A^r)\Phi(A)^{1-\a}|^2 \le W_2\Phi(A^{1+p+r})W_2^*
  \end{equation}
and combining (1.4), (1.5) and (1.6) completes the proof.
\end{proof}

\section{Matrix monotony inequalities}

\medskip

 Here we try to understand the results of the first section, using the
more general notion of a monotone pair. Recall that $(A,B)$ is said
to be a monotone pair in $\Mn$ if there exist a positive element
$C\in \Mn$ and two non-negative, non-decreasing functions $f$ and
$g$ so that $A=f(C)$ and $B=g(C)$. A typical example is $(A^p,A^q)$
for $A\geq 0$ and $p,\, q\geq 0$.

For technical reasons, we have to stick to $\Mn$, as many arguments
rely on the min-max principle. For instance, we use the following
result of
 \cite{Bo1} which compares the singular values of $AEB$ and $ABE$ for some
projections $E$.

 \begin{theorem}
   Let $(A,B)$ be a monotone pair  and let $E$ be a self-adjoint projection.
 Then, for some unitary $V$,
$$
|AEB| \le V|ABE|V^*.
$$
\end{theorem}

\medskip

As consequences, we have the following  Chebyshev's type eigenvalue
inequalities for compressions  \cite{Bo1},
$$
\lambda_j[(EAE)(EBE)]\le
\lambda_j[EABE]
$$
and
\begin{equation}
\lambda_j[(EAE)(EBE)(EAE)]\le \lambda_j[EABAE]
\end{equation}
where $\lambda_j[\cdot]$ stands for the list of eigenvalues arranged
in decreasing order with their multiplicities. Let $\Phi:\Mn\to \Md$
be a  unital completely positive (linear) map. It is well known
(Stinespring) that $\Phi$ can be decomposed as $\Phi(A)=E\pi(A)E$,
where $\pi:\Mn\to \mathbb M_{m}$ is a $*$-representation (with
$m\leq n^2d$) and $E\in \mathbb M_{m}$ is a rank $d$ projection (and
identifying
 $E\mathbb M_{m}E$ with $\Md$).
 Taking into account that we start from a commutative $C^*$-algebra, (2.1) is then equivalent to :

\smallskip

 \begin{cor}
Let $(A,B)$ be a  monotone pair in $\Mn$ and let $\Phi:\Mn\to \Md$
 be a unital positive map. Then, for some  unitary $V\in \Md$,
$$
\Phi(A)\Phi(B)\Phi(A) \le V\Phi(ABA)V^*.
$$
\end{cor}

\smallskip

In the case of pairs of positive powers $(A^p,A^q)$, such results
are easy consequences of Furuta's inequality. To apply Corollary 2.2
we define a special class of monotone pairs (of positive operators).

\medskip\noindent {\bf Definition.} A monotone pair $(A,B)$ is concave if
$A=h(B)$ for some concave function $h:[0,\infty)\to [0,\infty)$.

\medskip This class contains pairs of powers $(A^p,A^q)$ with $0\le
p\le q$ and we note that Corollaries 2.4-2.5 below are variations of
Theorem 1.1. We first state a factorization result.

 \begin{theorem}
Let $(A,B)$ be a concave monotone pair in $\Mn$ and let $\Phi:\Mn\to
\Md$ be a unital positive map. Then, for some contraction $K$ and
unitary $U$ in $\Md$,
$$
\Phi(B)\Phi(A) = \sqrt{\Phi(AB)}K\sqrt{\Phi(AB)}U.
$$
\end{theorem}

\smallskip

 \begin{proof} By a continuity argument we may assume that $A$ is
 invertible, hence
 \begin{equation*}
 \begin{pmatrix}
 AB & B\\ B & BA^{-1}
 \end{pmatrix}
 \ge 0.
 \end{equation*}
 Replacing $\Phi$ by  $\Phi\circ \mathbb E$, where $\mathbb E$ is the conditional
expectation onto the $C^*$-algebra generated by $A$ and $B$, we can
assume that $\Phi$ is completely positive so that we get
 \begin{equation*}
 \begin{pmatrix}
 \Phi(AB) & \Phi(B)\\ \Phi(B) & \Phi(BA^{-1})
 \end{pmatrix}
 \ge 0,
 \end{equation*}
 equivalently,
  \begin{equation*}
 \begin{pmatrix}
 \Phi(AB) & \Phi(B)\Phi(A)
 \\ \Phi(A)\Phi(B) & \Phi(A)\Phi(BA^{-1})\Phi(A)
 \end{pmatrix}
 \ge 0.
 \end{equation*}
 The concavity assumption on $(A,B)$ implies that $(A, BA^{-1})$ is
 a monotone pair, indeed both $h(t)$ and $t/h(t)$ are non-decreasing. By Corollary 2.2, we then have a unitary $U$ such that
 \begin{equation}
 \begin{pmatrix}
 \Phi(AB) & \Phi(B)\Phi(A)
 \\ \Phi(A)\Phi(B) & U^*\Phi(AB)U
 \end{pmatrix}
 \ge 0,
 \end{equation}
 equivalently, $$\Phi(B)\Phi(A)=
 \sqrt{\Phi(AB)}LU^*\sqrt{\Phi(AB)}U$$ for some contraction $L$.
 \end{proof}

\medskip

 Theorem 2.3 is equivalent to positivity of the
block-matrix (2.2). Considering the polar decomposition
$\Phi(A)\Phi(B)=W|\Phi(A)\Phi(B)|$ we  infer
\begin{equation*}
\begin{pmatrix}
I & -W^*
\end{pmatrix}
  \begin{pmatrix}
 \Phi(AB) & \Phi(B)\Phi(A)
 \\ \Phi(A)\Phi(B) & U^*\Phi(AB)U
 \end{pmatrix}
   \begin{pmatrix}
I \\ -W
\end{pmatrix}\ge 0
\end{equation*}
and thus obtain:

 \begin{cor}
Let $(A,B)$ be a concave monotone pair in $\Mn$ and let $\Phi:\Mn\to
\Md$ be a unital positive map. Then, for some unitary $V\in \Md$,
$$ |\Phi(A)\Phi(B)| \le \frac{\Phi(AB)+V\Phi(AB)V^*}{2}.
$$
\end{cor}

\smallskip
 Recall that a norm is said symmetric whenever $\Vert
UAV\Vert = \Vert A \Vert$ for all $A$ and all unitaries $U,V$.
Corollary 2.4 yields for {\it concave} monotone pairs some
Chebyshev's type inequalities for symmetric norms,
$$
\Vert \Phi(A)\Phi(B) \Vert \le \Vert \Phi(AB) \Vert.
$$ It is not clear that this can be extended to all monotone pairs.
In fact, for concave monotone pairs, Theorem 2.3 entails a stronger
statement. Given $X,Y\ge 0$, recall that the weak log-majorization
relation $X\prec_{\mathrm{ wlog}} Y$ means
$$
\prod_{j\le k} \lambda_j[X] \le \prod_{j\le k} \lambda_j[Y]
$$
for all $k=1,2,\cdots$. This entails $\Vert X\Vert \le \Vert Y\Vert$
for all symmetric norms. Theorem 2.3 and Horn's inequality yield :

\smallskip

 \begin{cor}
Let $(A,B)$ be a concave monotone pair in $\Mn$ and let $\Phi:\Mn\to
\Md$ be a unital, positive linear map. Then,
$$
|\Phi(A)\Phi(B)| \prec_{\mathrm{ wlog}} \Phi(AB).
$$
\end{cor}

\medskip

 In case of pairs $(A^p,A^q)$ we have more :

\begin{prop} Let $A\ge 0$ in $\Mn$, let $p,\, q\ge 0$
and let $\Phi$ as above. Then, for all eigenvalues,
$$
\lambda_j[\Phi(A^p)]\,\lambda_j[\Phi(A^q)] \le
\lambda_j[\Phi(A^{p+q})].
$$
\end{prop}

\begin{proof}
We outline an elementary proof. It suffices to show that for a given
projection $E$,
\begin{equation}
\lambda_j[EA^pE]\,\lambda_j[EA^qE] \le \lambda_j[EA^{p+q}E].
\end{equation}
In case of the first eigenvalue, this can be written via the
operator norm $\Vert\cdot\Vert_{\infty}$ as
\begin{equation}
\Vert EA^pE\Vert_{\infty} \Vert EA^qE\Vert_{\infty} \le \Vert
EA^{p+q}E\Vert_{\infty}.
\end{equation}
The proof of (2.4) follows from Young's and Jensen's inequalities
(always true for the operator norm),
\begin{align*}
\Vert EA^pE\Vert_{\infty} \Vert EA^qE\Vert_{\infty}
&\le\frac{p}{p+q}\Vert EA^pE\Vert_{\infty}^{\frac{p+q}{p}} +
\frac{q}{p+q}\Vert EA^qE\Vert_{\infty}^{\frac{p+q}{q}} \\
&\le \Vert EA^{p+q}E\Vert_{\infty}.
\end{align*}
The min-max characterization of eigenvalues combined with (2.4)
implies the proposition; indeed simply take $Q$ a projection commuting with
 $E$ of corank $j-1$ so that
 $\Vert QEA^{p+q}EQ\Vert_{\infty}=\lambda_j[EA^{p+q}E]$ and apply (2.4)
with $QE$ instead of $E$.
\end{proof}

\medskip
Results of this section follow from (2.1), equivalently from
Corollary 2.2, and hence have been stated for unital positive maps.
In fact these results can be stated to all sub-unital positive maps.
In particular the key Corollary 2.2 becomes :

\medskip\noindent {\bf Corollary 2.2a.} {\it Let $(A,B)$ be a  monotone
pair in $\Mn$ and let $\Phi:\Mn\to \Md$
 be a sub-unital positive map. Then, for some  unitary $V\in \Md$,
$$
\Phi(A)\Phi(B)\Phi(A) \le V\Phi(ABA)V^*.
$$
}

\begin{proof} Let ${\mathcal{A}}$ be the unital $*$-algebra generated
by $A$ and $B$. Restricting $\Phi$ to ${\mathcal{A}}$, it follows
from Stinespring's theorem (or from Naimark's theorem) that $\Phi$
can be decomposed as $\Phi(A)=Z\pi(A)Z$, where $\pi:{\mathcal{A}}\to
\mathbb M_{m}$ is a $*$-representation (with $m\leq nd$) and $Z\in
\mathbb M_{m}$ is a positive contraction (and identifying
 $E\mathbb M_{m}E$ with $\Md$ for some projection $E\ge Z$).

 Since $(\pi(A),\pi(B))$ is monotone, it then suffices to prove the result for
 congruence  maps of $\Mn$ of type $\Phi(X)=ZXZ$ where $Z$ is a positive contraction.
 We may then derive the result from (2.1) and a two-by-two trick :
 Note that
$$
A_0=\begin{pmatrix} A&0 \\ 0&0 \end{pmatrix} \quad {\mathrm {and}}
\quad B_0= \begin{pmatrix} B&0 \\ 0&0 \end{pmatrix}
$$
form a monotone pair. Note also that
$$
E=\begin{pmatrix} Z&(Z(I-Z))^{1/2} \\ (Z(I-Z))^{1/2}& I-Z
\end{pmatrix}
$$
is a projection. By (2.1),
\begin{equation*}
\lambda_j[(EA_0E)(EB_0E)(EA_0E)]\le \lambda_j[EA_0B_0A_0E],
\end{equation*}
equivalently,
\begin{equation}
\lambda_j[\,|EA_0EB_0^{1/2}|^2\,]\le
\lambda_j[(A_0B_0A_0)^{1/2}E(A_0B_0A_0)^{1/2}].
\end{equation}
Observe that
\begin{equation}
|EA_0EB_0^{1/2}|^2=\begin{pmatrix} B^{1/2}ZAZAZB^{1/2} & 0
\\ 0&0\end{pmatrix}\simeq
\begin{pmatrix} Z^{1/2}AZBZAZ^{1/2} & 0
\\ 0&0\end{pmatrix}
\end{equation}
where $\simeq$ means unitary equivalence, and similarly,
\begin{equation}
(A_0B_0A_0)^{1/2}E(A_0B_0A_0)^{1/2}\simeq\begin{pmatrix}
Z^{1/2}ABAZ^{1/2} & 0
\\ 0&0\end{pmatrix}.
\end{equation}
Combining (2.6) and (2.7) with (2.5) and replacing $Z^{1/2}$ by $Z$
yields $$ZAZ\cdot ZBZ\cdot ZAZ \le V(ZABAZ)V^*$$ for some unitary
$V$.
\end{proof}

 We end Section 2 with a remark about the two-by-two
trick used to derive Corollary 2.4. This can be used to get some
triangle type matrix inequalities. For instance, given two operators
$A$ and $B$ in some von Neumann algebra $\M$, there exists a partial
isometry $V$ such that:
\begin{equation}
|A+B|\le \frac{|A|+|B| +V^*(|A^*|+|B^*|)V}{2}.
\end{equation}
To check it, note
that, since for all $X$,
\begin{equation*}
  \begin{pmatrix}
 |X^*| & X
 \\ X^* & |X|
 \end{pmatrix}
  \ge 0,
  \end{equation*}
we thus have for all $V$,
\begin{equation*}
\begin{pmatrix}
-V^* I
\end{pmatrix}
  \begin{pmatrix}
 |A^*|+ |B^*|& A+B
 \\ A^*+B^* & |A|+|B|
 \end{pmatrix}
   \begin{pmatrix}
-V \\ I
\end{pmatrix}\ge 0,
\end{equation*}
and taking $V$ the partial isometry in the polar decomposition of
$A+B$ yields (2.8). This can be used to give a very short proof of
the triangle inequality for the trace norm in semi-finite von
Neumann algebras.

 \section{Means and order preserving relations}

Furuta's inequality was used as key tool in the first section; here
we present a possible proof for completeness.
 For that purpose we use the geometric mean of positive definite matrices
and Ando-Hiai's inequality. We do not pretend to originality and we
closely follow an approach due to Ando, Hiai, Fujii and Kamei.
However, we point out an interesting observation connecting the
geometric mean to complex interpolation. In fact this observation is
rather old : Identifying positive operators with quadratic forms, it
is worth noting that Donoghue's construction with complex
interpolation \cite{Do} seems to be the first appearance of the
matrix geometric mean.

In the whole section we consider $\B(\H)$, the set of  all bounded
operators on a  Hilbert space $\H$,  and its positive invertible
part, $\B^+$.

 For details and some  important results around the geometric mean we
refer to \cite{A}, \cite{AH}, references herein, and  \cite{Bha} for
a nice survey of other features of the weighted geometric means,
especially as geodesics on the cone of positive operators.

\subsection{Means and interpolation}

Let $\alpha\in[0,1]$ and consider a map
$$\begin{array}{clc}
\B^+\times \B^+ &\to & \B^+ \\
(A,B) &\mapsto &A\sharp_\alpha B\end{array}$$
satisfying the two natural requirements for an $\alpha$-geometrical mean
\begin{enumerate}[1.]
\item If $AB=BA$  then $A\sharp_\alpha  B=A^{1-\alpha}B^\alpha $.
\item $(X^*AX)\sharp_\alpha  (X^*BX)=X^*(A\sharp_\alpha B)X$
for any invertible $X$.
\end{enumerate}

Choosing the appropriate $X$, we necessarily have
\begin{equation*}
A\sharp_\alpha B  =A^{1/2}{(A^{-1/2}BA^{-1/2})}^{\alpha }A^{1/2}.
\end{equation*}
So there is a unique extension of the $\alpha$-geometrical mean for
commuting operators which is invariant under congruence, that is called
the $\alpha$-geometrical mean.

Matrix geometric means have their roots in the work of Pusz and
Woronowicz \cite{PW} about functional calculus for sesquilinear
forms. Their construction is closely related to complex
interpolation. Coming back to means, these links are even clearer.

We briefly recall the complex interpolation method of Calderon, see \cite{BL}
for a complete exposition.

Two Banach spaces $A_0$ and $A_1$ are said to be an interpolation
couple if there is another Banach space $V$ and continuous embeddings
$A_i\to V$.  So we have a way to identify elements and it makes sense
to speak of $A_0\cap A_1$ and $A_0+A_1$ (which are also Banach spaces
with the usual norms).  The idea of interpolation is to assign for
each $\alpha \in [0,1]$ a space that is intermediate between the
$A_i$. The construction is a bit technical.

Let $\Delta=\{ z\in \C \,|\, 0<{\rm Re} \,z<1\}$, $\delta_i=\{ z\in
\C \,|\, {\rm Re}\, z=i\}$ for $i=0$, 1. Define $\mathcal
F(A_0,A_1)$ as the set of maps $f:\overline \Delta \to A_0+A_1$,
such that
\begin{enumerate}[i)]
\item $f$ is analytic in $\Delta$.
\item for $i=0$, 1, $f(\delta_i)\subset A_i$ and
$f:\delta_i\to A_i$ is bounded and continuous.
\item for $i=0$, 1, $\lim_{t\in \R\to \pm\infty} ||f(i+{\rm
i}t)||_{A_i}=0$.
\end{enumerate}
Equipped with the norm
$$||f||=\max_{i=0,\,1}\{ \sup_{z\in \delta_i} ||f(z)||_{A_i}\}$$
$\mathcal F(A_0,A_1)$ becomes a Banach space. Finally for $\alpha \in [0,1]$,
$$(A_0,A_1)_\alpha=\{ x\in A_0+A_1 \,:\, \exists f\in \mathcal
F(A_0,A_1) \textrm{ so that} f(\alpha)=x\,\} $$
with the quotient norm
$$||x||_{(A_0,A_1)_\alpha}=\inf \{ ||f|| \,:\, f(\alpha)=x\,\}.$$

This functor has many nice properties. The most common is the
interpolation principle; consider two interpolation couples
$(A_0,A_1)$ and $(B_0,B_1)$ and bounded maps $T_i:A_i\to B_i$ so that
$T_1$ and $T_2$ coincide on $A_0\cap A_1$, then one can define a map
$T_\alpha:(A_0,A_1)_\alpha\to(B_0,B_1)_\alpha$ which extends $T_i$ on
$A_0\cap A_1$, and moreover one has $||T_\alpha||\leq ||T_0||^{1-\alpha}
||T_1||^\alpha$.

\smallskip

There are concrete examples where these interpolated norms can be
computed. If $A_0=L_\infty([0,1])$ and $A_1=L_1([0,1])$, one has
$(A_0,A_1)_\alpha=L_{1/\alpha}([0,1])$.

\smallskip

Using basic properties of the interpolation, it is easy to see that
the interpolation of two compatible Hilbert spaces is still a
Hilbert space. Indeed, by \cite{BL} Theorem 5.1.2, the parallelogram
identity
 is preserved by the complex interpolation method.

Let $A\in \B^+$, then it defines an equivalent hilbertian norm on
$\H$ by $||h||_{A}=||A^{1/2}h||_{\H}$. And conversely
any equivalent hilbertian norm on $\H$ arises from some $A\in \B^+$.
We denote by $\H_A$ the Hilbert space coming from $A$.

Now take $A_i\in \B^+$, $(\H_{A_0}, \H_{A_1})$ forms an
interpolation couple of Hilbert space (with the obvious
identification).  The resulting interpolated space for $\alpha\in
[0,1]$ will also give an equivalent norm on $\H$, associated to an
operator that we call $A_\alpha$. Let's have a look at the
properties of $(A_0,A_1)\mapsto I_\alpha(A_0,A_1)=A_\alpha$.

First, it is an easy exercise to check that if $A_0$ and $A_1$
commute then
$I_\alpha(A_0,A_1)=A_\alpha=A_0^{1-\alpha}A_1^{\alpha}$.

Secondly, let $X\in \B(\H)$ be invertible. With $B_i=X^*A_iX$, it is
clear that $X:\H_{B_i}\to \H_{A_i}$ is a unitary for $i=0$, 1. From
the interpolation principle, $X$ will also be unitary for the
interpolated norms. Coming back to operators, this says that
$I_\alpha(X^*A_0X,X^*A_1X)=X^*I_\alpha(A_0,A_1)X$.

So we can conclude that the $\alpha$-geometric mean is the
interpolation functor of index $\alpha$. With this in mind, all properties
of the means come from basic results in the complex interpolation theory.

\smallskip

 Take $(A_0,A_1)$ and $(B_0,B_1)$ in $(\B^+)^2$ and assume that $B_i\leq
A_i$. This means that the identity of $\H$ is a contraction from
$\H_{A_i}$ to $\H_{B_i}$.  By the interpolation principle, the same
holds for the interpolated norms. So we can conclude that the
$\alpha$-mean is monotone. Note that this gives another proof of the
monotony of $A\mapsto A^{\alpha}$ for $0\leq \alpha\leq 1$.

\smallskip

To get concavity of the mean is also easy for people familiar with
 interpolation.  Take $A_i$ and $B_i$ in $\B^+$, and $0<\lambda<1$ and
 notice that the map $\H_{\lambda A_i+ (1-\lambda) B_i}\to \H_{\lambda
 A_i}\oplus_2 \H_{(1-\lambda)B_i}$, $h\mapsto (h,h)$ is an
 isometry. From properties of the interpolation functor, we deduce
 that the same map $\H_{(\lambda A_0+ (1-\lambda)
 B_0)\sharp_\alpha(\lambda A_1+ (1-\lambda) B_1)}\to \H_{\lambda
 (A_0\sharp_\alpha A_1)}\oplus_2 \H_{(1-\lambda) (B_0\sharp_\alpha
 B_1)}$ is a contraction. Coming back to an inequality on operators
 gives the concavity. This  illustrates  the well-known fact
 that taking subspaces and interpolation do not commute.

Another useful result is the reiteration theorem : For any
$\alpha,\, \beta, \gamma\in [0,1]$, we have (provided that $A_0\cap
A_1$ is dense in both $A_0$ and $A_1$)
$$((A_0,A_1)_\alpha,
(A_0,A_1)_\beta)_\gamma=(A_0,A_1)_{(1-\gamma)\alpha+\gamma\beta}.$$
This means that for any $x,\,y,\,z\in [0,1]$ and $A,\, B\in \B^+$,
$$(A\sharp_x B)\sharp_z (A\sharp_y B)=A\sharp_{x(1-z)+yz}B.$$
Of course this can also be checked directly from the  formulae
defining  $\sharp$.

 The next theorem is  the Ando-Hiai inequality. We only use
the language of operator mean, but this is really a proof in the
spirit of the interpolation theory.

\begin{theorem} Let $A, B\in \B^+$ and $0<s<1$. Then,
$$
\Vert (A\sharp_{\alpha} B)^s\Vert_{\infty} \le \Vert
A^s\sharp_{\alpha} B^s\Vert_{\infty}.
$$
\end{theorem}

\begin{proof} By homogeneity we may assume $\Vert A\sharp_{\alpha} B\Vert_{\infty}=
1$. Hence we have $A\sharp_{\alpha} B\le I$. By using monotony of
geometric means and the reiteration principle we then get
\begin{align*}
A^s\sharp_{\alpha} B^s &= (I\sharp_s A) \sharp_{\alpha} (I\sharp_s B) \\
&\ge ((A\sharp_{\alpha} B)\sharp_s A)\sharp_{\alpha} ((A\sharp_{\alpha} B)\sharp_s B) \\
&= ((A\sharp_{\alpha} B)\sharp_s(A\sharp_0 B))\sharp_\alpha ((A\sharp_{\alpha} B)\sharp_s (A\sharp_1B)) \\
&= (A\sharp_{\alpha(1-s)}  B)\sharp_{\alpha} (A\sharp_{\alpha(1-s)+s}  B) =A\sharp_{\alpha} B. \\
\end{align*} Thus $\Vert
A^s\sharp_{\alpha} B^s\Vert_{\infty} \ge 1$ and this proves the
theorem.
\end{proof}

\medskip

\noindent{\bf Remark}. A theory of complex interpolation for
families of Banach spaces has been developed in \cite{CCRSW}. The
family may be indexed  by the unit circle in $\C$, say $A(z)$, with
some technical assumptions. The interpolation then provides a family
of spaces $A(z)$ for $|z|<1$. This can be used define a mean of
several operators. For instance, in the case of $n$ operators, one
may pick a  partition of the unit circle in $n$ sets $E_i$ with
Lebesgue measure $\alpha_i$, and choose the family $A(z)=\H_{A_i}$
if $z\in E_i$. Then the interpolated space at $0$ is of the form
$A(0)=\H_A$, and one may think of $A$ as a $(\alpha_i)$-mean of the
$A_i$'s. Unfortunately this definition depends on the choice of the
$E_i$ (unless $n=2$). This kind of approach for interpolation of a
finite family of spaces can also be found in \cite{Fa}.

\subsection{From means to order relations}
Next we explain how to go from Ando-Hiai's inequality to Furuta's
theorem (their equivalence was pointed out in \cite{FK}).

 Let  $A,\, B\in \B^+$ with $A\ge B$. Then, $A^{-1}\s_{1/2} B\le
I$, so by Ando-Hiai's inequality, $A^{-p}\s_{1/2} B^p\le I$ whenever
$p\ge 1$. Equivalently we have an order preserving relation for
$f(t)=t^p$ with  $p\ge 1$,
$$ A^p\ge (A^{p/2}B^pA^{p/2})^{1/2}, \quad p\ge 1.$$
Such inequalities suggest to look for the best exponents $p, r, w$ for which
\begin{equation}
A\ge B\ge 0 \quad \Rightarrow \quad  A^{(p+r)w}\ge
(A^{r/2}B^pA^{r/2})^w
\end{equation}
and consequently to get interesting substitutes to the lack of
operator monotony of $f(t)=t^p$, $p\ge 1$.

  To do so, it seems natural to find
relations for weighted geometric means of the form
\begin{equation}\label{asb}
A\ge B\ge 0 \quad  \Rightarrow \quad  A^{-r}\s_{\a}B^p\le I.
\end{equation}
Because of homogeneity, this can hold only for $\alpha=\frac {r}{p+r}$.
If $p\leq 1$, this inequality is obvious
 by the monotony of the mean. For $p>1$,
as above thanks to Ando-Hiai's inequality, one only need to find
$s\leq 1$ so that  $A^{-s r}\s_{\a}B^{sp}\le I$; we've just said
that $s=1/p$ works. We have proved :
\begin{lemma} Let  $A,\, B\in \B^+$ with $A\ge B$ and $p,r>0$. Then,
$$
 A^{-r}\sharp_{\frac{r}{p+r}} B^p\le I.
$$
\end{lemma}
Taking another mean with $B^p$, we obtain the optimal form of
(\ref{asb}) :
 \begin{lemma}
 Let $A,\, B\in \B^+$ with $A\ge B$ and $r>0$, $p\ge 1$. Then,
$$
 A^{-r}\sharp_{\frac{1+r}{p+r}} B^p\le B \le A.
$$
\end{lemma}

\noindent Hence we have recaptured quite easily two lemmas due to
Fujii and Kamei \cite{FK}.

  We come back to relations of the form
(3.1), the last lemma says :
\begin{equation}
A\ge B\ge 0 \quad \Rightarrow \quad A^{1+r} \ge
(A^{\frac{r}{2}}B^pA^{\frac{r}{2}})^{\frac{1+r}{p+r}},\qquad r>0,\
p\ge 1.
\end{equation}
Equivalently,
$$
A^{(p+r)w} \ge (A^{\frac{r}{2}}B^pA^{\frac{r}{2}})^{w},\qquad r>0,\
p\ge 1,
$$ where $w=\frac{1+r}{p+r}$. This is still valid for
$w\le\frac{1+r}{p+r}$ by the operator monotony of $t\mapsto
t^\alpha$, $0\leq \alpha\leq 1$. We obtain Furuta's theorem :
\begin{theorem} Let $A,\, B\ge 0$ in $\B(\H)$ and $r\ge 0$, $p\ge 1$. If
  $q\ge (p+r)/(1+r)$, then
$$
A^{(p+r)/q} \ge (A^{\frac{r}{2}}B^pA^{\frac{r}{2}})^{1/q}.
$$
\end{theorem}
\smallskip
\noindent The general statement follows from the case $B\in \B^+$ by
continuity.

\smallskip

\noindent \subsubsection{Comments}   Around 1985 it was conjectured
by Kwong that $A\ge B\ge 0$ entails $A^2\ge (AB^2A)^{1/2}$,
equivalently $A^2\ge |BA|$. In 1987, Furuta \cite{F} proved his
inequality. Some numerical experiments lead him to know the
condition on the exponents and he  obtained a direct, quite
ingenious proof. However the natural conjecture of Kwong may be
written via geometric means and is nicely answered by a basic case
of Ando-Hiai's inequality (1994). Hence order preserving relations
may be obtained from a study of weighted geometric means. We have
followed this idea, mainly developed by Ando, Hiai, Fujii and Kamei.

\bigskip

\bibliographystyle{amsplain}

\begin{thebibliography}{10}

\bibitem {A} T.\ Ando, \textit{Concavity of certain maps on positive definite matrices and applications to Hadamard products},
Linear Alg.\ Appl., \textbf {26} (1979), 203-241.

\bibitem {AH} T.\ Ando and F.\ Hiai, \textit{Log majorization and complementary Golden-Thompson type inequality},
Linear Alg.\ Appl., \textbf {197} (1994), 113-131.

\bibitem {BL} J. Bergh and  J. L\"ofstr\"om, {\it Interpolation spaces. An introduction.} Grundlehren der Mathematischen Wissenschaften, No. {\bf 223}.
Springer-Verlag, Berlin-New York,  1976.

\bibitem{Bha}  R. Bhatia, {\it Positive definite matrices.}
Princeton Series in Applied Mathematics. Princeton University Press, Princeton, NJ,  2007.

\bibitem {Bo1} J.-C.\ Bourin, \textit{Singular values of compressions, restrictions and
dilations}, Linear Alg.\ Appl., \textbf {360} (2003), 259-272.

\bibitem {C} M.-D.\ Choi, \textit{ A Schwarz inequality for positive linear maps on $C\sp{\ast} $-algebras},  Illinois J. Math.  {\bf 18}  (1974), 565-574.

\bibitem {CCRSW} R. R. Coifman, M. Cwikel, R. Rochberg, Y. Sagher, G.  Weiss, \textit{ A theory of complex interpolation for families of Banach spaces.}  Adv. in Math.  \textbf{43}  (1982),  no. 3, 203--229.


\bibitem {Davis} C.\ Davis, \textit{A Schwarz inequality for convex operator functions}, Proc. Amer. Math. Soc. {\bf 8} (1957), 42-44.

\bibitem {Do} W.\ F.\ Donoghue, \textit{The interpolation of quadratic norms}, Acta. Math.  {\bf 118} (1967), 251-270..


\bibitem{Fa}  A. Favini, {\it Su una estensione del metodo d'interpolazione complesso.}  Rend. Sem. Mat. Univ. Padova  {\bf 47}  (1972), 243-298.

\bibitem{FK} M. Fujii and E. Kamei, \textit{Ando-Hiai inequality and Furuta inequality
}, Linear Alg.\ Appl., \textbf{416} (2006) 541-545.

\bibitem{F} T. Furuta, \textit{$A\ge B\ge0$ assures $(B^rA^pB^r)^{1/q}\ge B^{(p+2r)/q}$ for $r\ge 0,p\ge0,q\ge1$ with $(1+2r)q\ge p+2r$},
Proc.\ Amer.\ Math.\ Soc. \textbf {101} (1987), 85-88.

\bibitem{K} R. Kadison, {\it A generalized Schwarz inequality and algebraic invariants for operator algebras}, Ann. of Math. (2) {\bf 56}, (1952). 494-503.

\bibitem{Paulsen} V. I. Paulsen, {\it Completely bounded maps and dilations}.
Pitman Research Notes in Mathematics Series, 146. Longman Scientific \& Technical, Harlow; John Wiley \& Sons, Inc., New York,  1986.

\bibitem{PW} W. Pusz and  S. L. Woronowicz, \textit{  Functional calculus for sesquilinear forms and the purification map},
 Rep. Mathematical Phys. \textbf{8}  (1975),  no. 2, 159-170.


\end{thebibliography}

\end{document}